\documentclass{amsart}
\usepackage{amsmath,amssymb,amsthm,url,comment}
\usepackage{graphicx}

\newtheorem{theorem}{Theorem}  
\newtheorem{lemma}[theorem]{Lemma}

\newtheorem{conjecture}{Conjecture} 
 
\newtheorem{corollary}[theorem]{Corollary}
\theoremstyle{definition}

\newtheorem{remark}[theorem]{Remark}
\newtheorem*{remark*}{Remark}

\newcommand{\Q}{\mathbb{Q}}

\newcommand{\Crossing}{
\raisebox{-2mm}{
\begin{picture}(24,24)
\put(0,0){\line(1,1){10}}
\put(14,14){\line(1,1){10}}
\put(0,24){\line(1,-1){24}}
\end{picture} } 
}

\newcommand{\smoothA}{
\raisebox{-2mm}{
\begin{picture}(24,24)
\qbezier(0,0)(14,12)(0,24)
\qbezier(24,0)(10,12)(24,24)
\end{picture} } 
}

\newcommand{\smoothB}{
\raisebox{-2mm}{
\begin{picture}(24,24)
\qbezier(0,0)(12,14)(24,0)
\qbezier(0,24)(12,10)(24,24)
\end{picture} } 
}

\begin{document}

\title[Finiteness of purely cosmetic surgery]{A remark on a finiteness of purely cosmetic surgeries}
\author{Tetsuya Ito}
\address{Department of Mathematics, Kyoto University, Kyoto 606-8502, JAPAN}
\email{tetitoh@math.kyoto-u.ac.jp}

\begin{abstract}
By estimating the Turaev genus or the dealternation number, which leads to an estimate of knot floer thickness, in terms of the genus and the braid index, we show that a knot $K$ in $S^{3}$ does not admit purely cosmetic surgery whenever $g(K)\geq \frac{3}{2}b(K)$, where $g(K)$ and $b(K)$ denotes the genus and the braid  index, respectively. In particular, this establishes a finiteness of purely cosmetic surgeries; for fixed $b$, all but finitely many knots with braid index $b$ satisfies the cosmetic surgery conjecture. 
\end{abstract} 

\maketitle

For a knot $K$ in the 3-sphere $S^{3}$ and $r \in \Q$, let $S^{3}_K(r)$ be the $r$-surgery on $K$. Two Dehn surgeries $S^{3}_K(r)$ and $S^{3}_K(r')$ on the same knot $K$ are \emph{purely cosmetic} if $r \neq r'$ but $S^{3}_K(r) \not \cong S^{3}_K(r')$. Here we denote by $M \cong N$ if $M$ and $N$ are \emph{orientation-preservingly} homeomorphic. 

\begin{conjecture}[Cosmetic surgery conjecture]
\label{conj:PCSC}
A non-trivial knot does not admit purely cosmetic surgeries.
\end{conjecture}

One must be careful that we are taking account of orientations; there are several examples of \emph{chirally cosmetic surgery}, a pair of Dehn surgeries on the same knot that yields orientation-reversingly homeomorphic 3-manifolds. For example, for the trefoil knot $K$, $S^{3}_K(9) \cong -S^{3}_K(\frac{9}{2})$ \cite{ma}. Here $-M$ is the 3-manifold $M$ with opposite orientation.

For a knot $K$ in $S^{3}$, let $g(K)$ be the genus and $b(K)$ be the braid index of $K$. The aim of this note is to point out the following finiteness result on purely cosmetic surgeries that gives a strong supporting evidence for Conjecture \ref{conj:PCSC}.
  
\begin{theorem}
\label{theorem:main}
If $g(K) \geq \frac{3}{2}b(K)$ then $K$ does not admit a purely cosmetic surgery.
In particular, for given $b>0$, these is only finitely many knots with braid index $b$ that admit purely cosmetic surgeries.
\end{theorem}
Here, the latter fininteness assertion follows from Birman-Menasco's finiteness theorem \cite{bm}: for given $g,b>0$ there are only finitely many knots with genus $g$ and braid index $b$.

Our proof of Theorem \ref{theorem:main} is based on its quantitative refinement of Birman-Menasco's finiteness theorem \cite{it} and the following quite strong constraint for purely cosmetic surgeries.

\begin{theorem}[Hanselman \cite{ha}]
\label{theorem:hanselman}
Let $K$ be a non-trivial knot and $th(K)$ be the Heegaard Floer thinkness of $K$. If  $S^{3}_K(r) \cong S^{3}_K(r')$ for $r\neq r'$, then either
\begin{itemize}
\item $\{r,r'\}=\{2,-2\}$ and $g(K)=2$, or,
\item $\{r,r'\}=\{\frac{1}{q},-\frac{1}{q}\}$ for some $0<q \leq \frac{th(K)+2g(K)}{2g(K)(g(K)-1)}$.
\end{itemize}
\end{theorem}
Here $th(K)$ is the thickness of the knot floer homology.

Thus when $g(K)\neq 2$ and $th(K)$ is small compared with $g(K)$ then $K$ does not admit purely cosmetic surgery. This motivates to study a relation between $g(K)$ and $th(K)$, in particular the (upper) bound of $\frac{th(K)}{g(K)}$. Here we give an upper bound of the thinkness $th(K)$ in terms of $g(K)$ and $b(K)$. 

Although our argument is applied for the case $b(K)=2,3$, we restrict our attention to the case $b(K) \geq 4$.

\begin{lemma}
\label{lemma:thickness}
If $b(K)\geq 4$,
\[ th(K) \leq \frac{1}{2}(2b(K)-5)(2g(K)-1+b(K))\]
\end{lemma}
\begin{proof}
For a knot diagram $D$, the \emph{Turaev genus} $g_T(D)$ is defined by
\[ g_T(D)=\frac{1}{2}(c(D)+2 -|s_A|-|s_B|) \]
where $c(D)$ is the crossing number of $D$, and $|s_A|$ and $|s_B|$ the number of circles obtained by $A$- (resp. $B$-) smoothing of crossings of $D$ given by
\[ \smoothA \stackrel{A}{\longleftarrow}\Crossing  \stackrel{B}{\longrightarrow} \smoothB. \]
The Turaev genus $g_T(K)$ of a knot $K$ is the minimum of $g_T(D)$ among the all diagrams $D$ of $K$.
In \cite{lo} Lowrance showed the inequality
\[ th(K) \leq g_T(K). \]
Dor any diagram $D$, $|s_A|,|s_B|\geq 1$ so $g_T(D) \leq \frac{1}{2}c(D)$. Hence we have a canonical upper bound of the Turaev genus
\begin{equation}
\label{eqn:Turaev-crossing}
g_T(K) \leq \frac{1}{2}c(K)
\end{equation}
Finally, by the quantitative Birman-Menasco finiteness Theorem\footnote{When $b(K)=2,3$ a similar inequality holds but the coefficient $(2b(K)-5)$ is $1$ and $\frac{5}{3}$, respectively.} \cite{it}, if $b(K)\geq 4$ we get
\[ c(K) \leq (2b(K)-5)(2g(K)-1+b(K)). \]
These three inequalities prove the desired inequality.
\end{proof}

\begin{proof}[Proof of Theorem \ref{theorem:main}]
In the following we assume that $b(K)\geq 4$ since Varvarezos proved the cosmetic surgery conjecture for the case $b(K)=3$ \cite{va}.
Also, we assume that $g(K)\neq 2$.

Assume, to the contrary that $K$ admits a purely cosmetic surgery. 
By Theorem \ref{theorem:hanselman}, such a knot must satisfy 
\[  1 \leq \frac{th(K)+2g(K)}{2g(K)(g(K)-1)} \iff 2g(K)(g(K)-2) \leq th(K)\]
so by Lemma \ref{lemma:thickness} we conclude when a knot $K$ admits a purely cosmetic surgery then it satisfies the inequality
\begin{align*}
& 2g(K)(g(K)-2) \leq \frac{2b(K)-5}{2}(2g(K)-1+b(K)) \\
\end{align*}
That is, we get a constraint for a knot $K$ to admit a purely comsetic surgery
\begin{equation}
\label{eqn:c}
4g(K)^2+(2-4b(K))g(K) +(2b(K)-5)(1-b(K)) \leq 0.
\end{equation}
Now the assertion of the theorem follows from an easy computation that $g(K)\geq \frac{3}{2}b(K)$ then the equation (\ref{eqn:c}) is never satisfied.
\end{proof}

As the proof indicates, our sufficient condition $g(K) \geq \frac{3}{2}b(K)$ can be improved if one can improve an estimate of $th(K)$ in Lemma \ref{lemma:thickness}.

\begin{remark}
Instead of using an obvious bound (\ref{eqn:Turaev-crossing}) of the Turaev genus, by using a different upper bound \cite[Corollary 7.3]{d+}
\[ g_T(K) \leq c(K)-\textrm{span}\,V_K(t)\]
where $V_K(t)$ denotes the Jones polynomial, we get a different constraint.
 if $K$ admits a purely cosmetic surgery, then
\begin{equation}
\label{eqn:c2}
2g(K)^2 + (6-4b(K))g(K)+(2b(K)-5)(1-b(K))+\textrm{span}\,V_K(t) \leq 0.
\end{equation}
\end{remark}

Here we give a mild improvement of Lemma \ref{lemma:thickness}.
For a diagram $D$ of a knot $K$, the \emph{dealternation number} $dalt(D)$ of $D$ is the minimum number of crossing change needed to make $D$ into an alternating diagram. The \emph{dealternation number} of a knot $K$ is the minimum of $dalt(D)$ among the all diagrams $D$ of $K$. 
It is known that $g_T(K)\leq dalt(K)$ \cite{ak}, so evaluating the dealternation number also gives an upper bound of the thickness.

We prove the following estimate of the dealternation number (hence the Turaev genus and the thickness) in terms of genus and braid index, which is interesting in its own right.

\begin{theorem}
\label{theorem:dalt}
If $b(K)\geq 4$ then 
\[ th(K) \leq g_T(K) \leq dalt(K) \leq (b(K)-3+\frac{1}{b(K)})(2g(K)-1+b(K)).\]
\end{theorem}

\begin{proof}
Let $n=b(K)$ and and $B_n$ be the braid group of $n$-strands. We denote the standard generators of $B_n$ by $\sigma_1,\ldots,\sigma_{n-1}$.
We say that a braid is \emph{alternating} if it is a product of $\{\sigma_1,\sigma_{2}^{-1}, \sigma_{3},\sigma_{4}^{-1},\ldots ,\sigma_{2i-1}, \sigma_{2i}^{-1},\ldots\}$. Obviously, the closure of an alternating braid is an alternating diagram.

 For $1\leq  i < j \leq n$, let $a_{i,j}$ be the band generator given by
\[ a_{i,j}=(\sigma_{i}\sigma_{i+1}\cdots \sigma_{j-1})\sigma_j(\sigma_{i}\sigma_{i+1}\cdots \sigma_{j-1})^{-1}\]
A band generator $a_{i,j}$ can be seen as the boundary a twisted band connecting the $i$-th and $j$-th strand of the braid.
Thus when $K$ is represented as the closure of a braid $\beta \in B_n$, by giving $\beta$ as a product of band generators, we get a Seifert surface $F_{\beta}$ of $K$ called the  \emph{Bennequin surface} associated to the braid (word) $\beta$.

First we treat the case $K$ bounds a minimum genus Bennequin surface of minimum braid index, that is, $K$ is represented by a closed $n$-braid $\beta$ such that its Bennequin surface $F_{\beta}$ is the minimum genus Seifert surface of $K$.

Thanks to the relation 
\[ \sigma_{j-1}\sigma_{j}^{\pm 1} \sigma_{j-1}^{-1} =\sigma_{j}^{-1}\sigma_{j-1}^{\pm 1} \sigma_{j}\]
by taking suitable word representative of $a_{i,j}^{\pm 1}$, each band generator $a_{i,j}$ except $a_{1,n}$ can be made so that it is alternating by changing at most $n-3$ crossings. The exceptional case $a_{1,n}^{\pm 1}$ can be made so that it is alternating by changing $n-2$ crossings. 
Thus 
\[ dalt(K) \leq \sum_{\substack{1\leq i<j \leq n \\ (i,j) \neq(1,n)}}(n-3)r_{i,j}  +  (n-2) r_{1,n} = \sum_{1\leq i<j \leq n }(n-3)r_{i,j}  + r_{1,n}\]
where $r_{i,j}$ is the number of $a_{i,j}^{\pm 1}$ in the braid $\beta$.

On the other hand, since we assume that the Bennequin surface $F_{\beta}$ associated with the $n$-braid $\beta$ has genus $g(K)$, 
\[ \sum_{1\leq i<j \leq n}r_{i,j} = 2g(K)-1+n. \]

Let $\delta=\sigma_{1}\sigma_2\sigma_{3}\cdots \sigma_{n-1}$. Since $\delta a_{i,j} \delta^{-1} = a_{i+1,j+1}$ (here we regard indices modulo $n$; for example, $\delta a_{1,n} \delta^{-1}=a_{2,n+1}$ is understood as $a_{1,2}$), by taking conjugates of $\delta$ if necessary, we may assume that $r_{1,n} \leq \frac{1}{n}(r_{1,2}+r_{2,3}+r_{3,4}+\cdots + r_{n-1,n}+r_{1,n}) \leq \frac{1}{n}(2g(K)-1+n)$.

Thus we conclude
\[ dalt(K) \leq \sum_{1\leq i<j \leq n }(n-3)r_{i,j}  + r_{1,n} \leq (n-3+\frac{1}{n})(2g(K)-1+n)\]
as desired.

Next we assume that $K$ does not bound a minimum genus Bennequin surface of the minimum braid index. To treat this case we quickly review a main strategy of the proof of quantitative Birman-Menasco theorem \cite{it}, how to relate the genus, braid index and crossing number (although we do not need to use or know the details).

We put a minimum genus Seifert surface $F$ of $K$ so that it admits a braid foliation. Let $R_{aa}, R_{ab}$ be the number of aa-tiles and ab-tiles of the braid foliation. What we showed in \cite{it} is two inequalities 
\begin{equation}
\label{eqn:cross} c(K) \leq (2n-5) R_{aa} + (n-3)R_{bb} 
\end{equation}
and
\begin{equation}
\label{eqn:genus}
2R_{aa}+R_{ab} \leq 2(2g(K)-1+b(K)).
\end{equation}
More precisely, the inequality (\ref{eqn:cross}) is obtained by observing that the braid foliation gives rise to an explicit closed $n$-braid representative $\beta$, such that one aa-tile provides a braid which is a band generator
\[ a_{i,j}^{\pm 1}, \qquad (i,j) \neq (1,n)\]
and that one ab-tile provides a braid of the form
\[ \gamma_{i,j}^{\pm 1}, \qquad |i-j| \leq n-3.\]
Here $\gamma_{i,j}$ denotes the braid
\[ \gamma_{i,j}=
\begin{cases}
\sigma_{i}\sigma_{i+1}\cdots \sigma_{j} & (i\leq j)\\
\sigma_{i}\sigma_{i-1}\cdots \sigma_{j} & (i\geq j).
\end{cases} \]

If $n$ is odd, then each braid $\gamma_{i,j}$ can be made into alternating by at most $\frac{1}{2}(n-3)$ crossing changes. Each band generator $a_{i,j}$ coming from aa-tile can be made into alternating by at most $(n-3)$ changes since $a_{1,n}$ does not appear. Therefore we conclude
\begin{align*}
dalt(K) & \leq (n-3)R_{aa} + \frac{1}{2}(n-3)R_{ab} = \frac{1}{2}(n-3)(2R_{aa}+R_{ab})\\
& \leq  (n-3)(2g(K)-1+n).
\end{align*}

If $n$ is even, let $M$ be the number $\gamma_{i,j}$ produced by ab-tiles such that 
$\gamma_{i,j}$ is made into alternating by $\frac{1}{2}(n-2)$ crossing changes. 
By taking the mirror image of $\beta$ if necessary, we may assume that $M \leq \frac{1}{2}R_{ab}$.
Since other braids $\gamma_{i,j}$ from ab-tiles can be made into alternating by at most $\frac{1}{2}(n-2)$ crossing changes, 
\begin{align*}
dalt(K)& \leq (n-3)R_{aa} + \frac{1}{2}(n-4)(R_{ab}-M) + \frac{1}{2}(n-2) M\\
& \leq (n-3)R_{aa} + \frac{1}{2}(n-3)R_{ab} = \frac{1}{2}(n-3)(2R_{aa}+R_{ab})\\
& \leq (n-3)(2g(K)-1+n).
\end{align*}
\end{proof}

Using this refinement we can improve a sufficient condition in Theorem \ref{theorem:main}. For example, for the case $b(K)=4$, a direct computation shows that 
\begin{corollary}
A knot $K$ with braid index $4$ does not admit purely cosmetic surgery if $g(K) \neq 2,3$.
\end{corollary}

\section*{Acknowledgement}
The author has been partially supported by JSPS KAKENHI Grant Number 19K03490,16H02145.


\begin{thebibliography}{1}
\bibitem{ak} T. Abe, and K. Kishimoto, 
{\em The dealternating number and the alternation number of a closed 3-braid,}
J. Knot Theory Ramifications 19 (2010), no. 9, 1157--1181.
\bibitem{bm} J. Birman, W. Menasco,   
{\em Studying links via closed braids. VI. A nonfiniteness theorem.} 
Pacific J. Math. \textbf{156} (1992), no. 2, 265-285.
\bibitem{d+} O. Dasbach, D. Futer, E. Kalfagianni, X.-S. Lin,  and N. Stoltzfus, 
{\em The Jones polynomial and graphs on surfaces.}
J. Combin. Theory Ser. B 98 (2008), no. 2, 384--399.

\bibitem{ha} J. Hanselman,
{\em Heegaard Floer homology and cosmetic surgeries in $S^3$.}
J. Eur. Math. Soc. to appear.
 
\bibitem{it} T. Ito,
{\em A quantitative Birman-Menasco finiteness theorem and its application to crossing number.}
arXiv:2010.12150v3.

\bibitem{lo} A. Lowrance, 
{\em On knot Floer width and Turaev genus.}
Algebr. Geom. Topol. 8 (2008), no. 2, 1141--1162.
 
\bibitem{ma} Y. Mathieu, 
{\em Closed 3-manifolds unchanged by Dehn surgery.}
J. Knot Theory Ramifications 1 (1992), no. 3, 279--296.

\bibitem{va} K. Varvarezos,
{\em 3-braid knots do not admit purely cosmetic surgeries,}
Acta Math. Hungar. (2021). doi.org/10.1007/s10474-020-01129-z.

\end{thebibliography}
\end{document}